\documentclass[11pt]{article}
\usepackage[margin=1.25in]{geometry} 
\usepackage{amsthm}
\usepackage{amssymb,amsmath,amsfonts,microtype}
\usepackage{graphicx}
\usepackage{pdfsync}
\usepackage{comment}

\usepackage{amssymb}
\usepackage[hidelinks]{hyperref}
\usepackage[capitalize]{cleveref}
\usepackage{microtype}
\usepackage{esint}

\theoremstyle{plain}
\newtheorem{thm}{Theorem}
\newtheorem{lem}{Lemma}[section]
\newtheorem{prop}{Proposition}[section]
\newtheorem{cor}{Corollary}[section]

\theoremstyle{definition}
\newtheorem{defn}{Definition}[section]

\newtheorem*{thmB}{Theorem B}
\newtheorem*{thmA}{Theorem A}
\newtheorem*{thmC}{Theorem C}
\newtheorem*{thmD}{Theorem D}

\newtheorem*{thmB'}{Theorem B$^\prime$}
\newtheorem*{thmC'}{Theorem C$^\prime$}

\theoremstyle{remark}

\newcommand{\R}{\mathbb{R}}

\newcommand{\Z}{\mathbb{Z}}
\newcommand{\N}{\mathbb{N}}
\newcommand{\F}{\mathbb{F}}
\newcommand{\PP}{\mathbb{P}}
\newcommand{\EE}{\mathbb{E}}

\newcommand{\La}{\Lambda}
\newcommand{\de}{\delta}

\newcommand{\De}{\Delta}

\newcommand{\eps}{\varepsilon}

\newcommand{\1}{\mathbf{1}}

\newcommand{\subs}{\subseteq}

\newcommand{\EXP}{\text{Exp}}

\numberwithin{equation}{section}
\newcommand{\eq}{\begin{equation}}
\newcommand{\ee}{\end{equation}}
\newcommand{\eqa}{\begin{eqnarray}}
\newcommand{\eea}{\end{eqnarray}}

\title{Polynomial progressions in the generalized twin primes}
\author{
Andrew Lott \and
Nagendar Reddy Ponagandla
}

\begin{document}

\maketitle

\begin{abstract} 
By Maynard's theorem and the subsequent improvements by the Polymath Project, there exists a positive integer $b\leq 246$ such that there are infinitely many primes $p$ such that $p+b$ is also prime. Let $P_1,...,P_t\in \Z[y]$ with $P_1(0)=\cdots=P_t(0)=0$. We use the transference argument of Tao and Ziegler to prove there exist positive integers $x, y,$ and $b \leq 246 $ such that $x+P_1(y),x+P_2(y),...,x+P_t(y)$ and $x+P_1(y)+b,x+P_2(y)+b,...,x+P_t(y)+b$ are all prime. Our work is inspired by Pintz, who proved a similar result for the special case of arithmetic progressions. 
\end{abstract}

\section{Introduction}

In their landmark paper \cite{GT06}, Green and Tao used the transference principle and Szemerédi's theorem to prove that the primes contain arbitrarily long arithmetic progressions. A key component of their proof was the construction of a pseudorandom measure inspired by the Goldston-Yildirim sieve \cite{GY05}. This groundbreaking result has since motivated numerous generalizations. Notably, Tao and Ziegler \cite{TZ08} showed that the primes also contain polynomial progressions of the form
\[
  x + P_1(y),\ x + P_2(y),\ \dots,\ x + P_t(y),
\]
where $P_1, \dots, P_t \in \mathbb{Z}[y]$ and $P_1(0) = \cdots = P_t(0) = 0$. One of the main ingredients in their proof was a relative PET induction scheme.

Since the release of \cite{GT06} and \cite{TZ08}, many simplifications have been developed. Notably, Gowers found a much shorter way to perform the transference step using polynomial approximation and the geometric Hahn–Banach theorem \cite{Gow10}, and Conlon-Fox-Zhao introduced a densification technique which further streamlined the transference step \cite{CFZ15}.

Incorporating these simplifications, in  \cite{Lott}, Lott-Magyar-Petridis-Pintz generalized the result of Tao and Ziegler to certain non-degenerate polynomial configurations in the prime lattice. Much of our current paper is based on \cite{Lott}.

Rather than seeking new patterns in the primes, one might instead ask for patterns in other sparse subsets of $\mathbb{N}$. Inspired by the approach of Green and Tao, Pintz used a transference argument to prove the following theorem in \cite{pintz1}.

\begin{thmA}
Let $k,j\in \N$. Let $\{h_1,...,h_k\}$ be an admissible $k$-tuple. Suppose there exists $\eps_0>0$ for which 
\[
|\mathcal{A}\cap [N]|\gg \frac{N}{\log^kN},
\]
where 
\[\mathcal{A}\subset \{n: P^{-}((n+h_1)\cdots(n+h_k))>n^{\eps_0}\}\footnote{For \( m \in \mathbb{Z} \), \( P^{-}(m) \) denotes the smallest prime divisor of \( m \).}.\]
Then $\mathcal{A}$ contains a $j$-term arithmetic progression. 
\end{thmA}
Recall that a set $\{h_1, . . .,h_k\}$ of distinct non-negative integers is admissible if, for
every prime $p$, there is an integer $a_p$ such that $a_p\not\equiv h \bmod{p}$ for all $h\in \{h_1,...,h_k\}$.

When $k = \epsilon_0 = 1$, Theorem A recovers the Green--Tao theorem. For $k = 2$ and $\epsilon_0 = 1/10$, this recovers a result of Zhou, which gives arithmetic progressions in the Chen primes \cite{Zhou}. (Zhou defined the Chen primes to be the set of primes $p$ such that $p+2$ is either a prime or a product of two primes and $P^{-}(p+2)\geq p^{1\slash 10}$.)   Historically, verifying that $|\mathcal{A} \cap [N]| \gg N / \log^k N$ posed a major challenge. (Here, $\mathcal{A}$ is defined as in Theorem A.) However, now we have the following version of Maynard’s theorem, which was proved by the Polymath Project (see Remark 32 in \cite{polymath8b}). 
\begin{thmB}[Maynard, Polymath]
Let \( m \in \mathbb{N} \). There exists \( k \) large enough (depending on \( m \)) such that for any admissible \( k \)-tuple \( \{h_1, \ldots, h_k\} \), there exists \( \varepsilon_0 > 0 \) such that
\[
|\mathcal{A} \cap [N]| \gg \frac{N}{\log^k N},
\]
where
\begin{align}\label{specialset}
\mathcal{A} = \{ n : &P^{-}((n + h_1)\cdots(n + h_k)) > n^{\varepsilon_0} \text{ and}\\
&\text{at least } m + 1 \text{ of } n + h_1, \ldots, n + h_k \text{ are prime} \}\notag.
\end{align}
\end{thmB}

The idea of using Theorem B came from a paper of Parshall \cite{Hans}, where he proved the existence of large configurations of irreducible polynomials over finite fields that are separated by low degree polynomials. With Theorem B, we are able to show the following.
\begin{thm}
Let $\mathcal{A}$ be defined as in \eqref{specialset}, and let $P_1,...,P_t\in \Z[y]$ with $P_1(0)=\cdots=P_t(0)=0$. Let $\delta>0$ and $N>N(\delta,P_1,...,P_t)$. If $A\subset \mathcal{A}\cap [N]$ with $|A| \geq \delta |\mathcal{A}\cap [N]|$, then there exist positive integers $x$ and $y$ such that $\{x+P_1(y),...,x+P_t(y)\}\subset A$. 
\end{thm}

By the pigeonhole principle, there exist $1\leq i_1<i_2<\cdots< i_{m+1}\leq k$ such that the set 
\begin{equation}\label{specialtuple}\{n:P^{-}((n+h_1)\cdots(n+h_k))>n^{\epsilon_0}\hbox{ and } n+h_{i_1},...,n+h_{i_{m+1}}\hbox{ are all prime}\}\end{equation}
has positive relative upper density in $\mathcal{A}$. \footnote{The relative upper density of $B\subset \mathcal{A}$ is defined by $\limsup_{N\to \infty}\frac{|B\cap [N]|}{|\mathcal{A}\cap [N]|}$. } Thus, we have the following corollary. 

\begin{cor}
Let $\mathcal{A}$ be defined as in \eqref{specialset} and $h_{i_1},...,h_{i_{m+1}}$ be defined as in \eqref{specialtuple}. Let $P_1,...,P_t\in \Z[y]$ with $P_1(0)=\cdots=P_t(0)=0$. There exist positive integers $x$ and $y$ such that $\{x+h_{i_{\ell}}+P_j(y): 1\leq \ell\leq m+1, 1\leq j\leq t\}$ are all prime. 
\end{cor}
In the case \( m = 1 \), the Polymath Project~\cite{polymath8b} constructed an admissible \( 50 \)-tuple \( \{h_1, \dots, h_{50}\} \) such that 
\[
\max\{ |h_i - h_j| : 1 \leq i < j \leq 50 \} = 246
\]
and such that the conclusion of Theorem B holds. From this, we obtain the following corollary.

\begin{cor}
There exists a positive integer $b\leq 246$ for which the following holds. Let $P_1,...,P_t\in \Z[y]$ with $P_1(0)=\cdots=P_t(0)=0$. There exist $x,y\in \N$ such that 
$$x+P_1(y),x+P_2(y),...,x+P_t(y)\hbox{ and } x+P_1(y)+b,x+P_2(y)+b,...,x+P_t(y)+b$$
are all prime. 
\end{cor}

\section{Notation}

Throughout the paper, $\ll$ will be the Vinogradov notation, $\displaystyle{\Z_{n} = \Z / n\Z}$, $[N] = \{1,\cdots,N\}$, and we use the convenient averaging notation: 
\[\EE_{x\in B}g(x)=\frac{1}{|B|}\sum_{x\in B}g(x).\]
We also use 
\[\exp(x)=e^x \hbox{ and }\EXP(x)=e^x-1.\]
In addition, for two integers $a,b$, let $[a,b]$ denote the least common multiple of $a$ and $b$. 

\section{Sketch of the proof}\label{Sketch}

Let $\mathcal{A}$ be defined as in \eqref{specialset} and suppose $A\subs \mathcal{A}\cap [N']$ with $|A|\geq \de\,|\mathcal{A}\cap [N']|$. We apply the transference principle to prove Theorem 1. Informally, the idea is that if $\mathcal{A}$ is pseudorandom in the appropriate sense, then our relatively dense subset $A \subset \mathcal{A}$ can be modeled by a dense subset $B \subset \mathbb{N}$. Since $B$ has positive density, it contains polynomial progressions by a theorem of Bergelson and Leibman \cite{BL96}. Hence, we expect $A$ to inherit similar structure. While the actual argument is more delicate, this heuristic provides a helpful mental picture.

With this heuristic in mind, we require an appropriate version of the polynomial extension of Szemer\'{e}di's theorem due to Bergelson-Leibman \cite{BL96}. Before stating this theorem, we need the following notation. Let $P_1,...,P_t\in \Z[y]$ and $f:[N]\to\R$. We define the normalized counting operator $\Lambda_{P_1,...,P_t}$ by
\[\Lambda_{P_1,....,P_t}f:=\EE_{y\in [M]}\EE_{x\in [N]}f(x+P_1(y))\cdots f(x+P_t(y)),\]
where $M=\log^LN.$ Notice that if $f=\1_B$ is the indicator function for $B\subset[N]$, then $\Lambda_{P_1,....,P_t}f$ gives a normalized count of polynomial progressions in $B$. Now we can state the relevant version of the polynomial Szemer\'{e}di theorem as given in \cite{TZ08,TZ14}. 
\begin{thmC} Let $\de>0$, $P_1,...,P_t\in \Z[y]$, $P_1(0)=\cdots=P_t(0)=0$. Let $N>N(\delta, P_1,...,P_t)$. If $g:[N]\to [0,1]$ is a function satisfying $\EE_{x\in [N]} g(x)\geq\de$, then one has 
\eq
\La_{P_1,...,P_t}g\geq c(\de)-o(1),
\ee
where $c(\de)>0$ is a constant that depends only on $\de$.
\end{thmC}

In Section \ref{Wtrick}, we use the $W$-trick to construct a function $f_A$ which is a scaled indicator function for $A$ on a unit congruence class $b\bmod W$, where $W$ is the product of small primes; essentially $f_A(x)$ is a constant multiple of $\1_A(Wx+b)$. This allows us to deal with the fact that the primes are not equidistributed across congruence classes mod $p$ for small primes $p$. Using Theorem~B and the pigeonhole principle, $A$ is already dense enough that after restricting to a suitable $b\bmod W$ and applying the natural normalization, we have $\EE_{x\in[N]}f_A(x)\ge \delta'>0$, where
$N=\lfloor N'/W\rfloor$.

Then we construct a pseudorandom measure $\nu_{\mathcal{A}}:\mathbb{N}\to \R_{\geq 0}$ which majorizes $f_A$. In Sections \ref{correlationestimate} and \ref{polynomialforms}, we prove $\nu_{\mathcal{A}}$ satisfies a pseudorandomness property called the polynomial forms condition, as defined in Lott-Magyar-Petridis-Pintz \cite{Lott}. In Section \ref{transferenceprinciple} we apply a transference theorem from Lott-Magyar-Petridis-Pintz \cite{Lott}. This allows us to model $f_A$ by a function $g:[N]\to[0,1]$ with $\EE_{x\in[N]}g(x)\ge \delta'/2$ and with essentially the same count of the polynomial configurations as given by $\Lambda_{Q_1,...,Q_t}g$. (After the $W$-trick we work with the rescaled polynomials $Q_j(y):=P_j(Wy)/W$.) At that point Theorem~C finishes the proof.

\section{The W-trick and the pseudorandom measure}\label{Wtrick}
Now we use the $W$-trick to construct $f_A$ and the pseudorandom measure $\nu_\mathcal{A}$. Our work in this section is inspired by Pintz \cite{pintz1}. However, we define our pseudorandom measure in a way that allows us to use Fourier analysis, which is similar to the approach of Tao and Ziegler \cite{TZ08,TZ14}.

Let $A\subset \mathcal{A}\cap [N']$ and $|A|=\delta|\mathcal{A}\cap [N']|$. Thus, by Theorem B, 
\begin{equation}
|A|\geq \delta \frac{C_1 N'}{\log^k N'}, \label{density}
\end{equation}
for some $C_1>0$. We begin by applying the $W$-trick in order to account for the fact that $\mathcal{A}$ is not equidistributed mod $p$ for any prime $p$. Let $w:=\log\log\log N'$, $W:=\prod_{p\le w}p$, and $N:=\lfloor N'/W\rfloor$. We will choose a residue class $b \pmod W$ with $(b,W)=1$ and $b\not\equiv -h_i \pmod p$ for all $p\mid W$ and $1\le i\le k$, such that $A$ has many elements of the form $Wx+b$ with $x\in[\sqrt N,\,N-\sqrt N]$. Assuming such a choice of $b$ for the moment, define
\begin{equation}\label{f_A}
f_A(x)=c_0\left(\frac{\phi(W)\log N}{W}\right)^k \mathbf 1_A(Wx+b)\mathbf 1_{[\sqrt{N},\,N-\sqrt{N}]}(x),
\end{equation}
where $c_0>0$ is a small constant to be fixed later. We will also show that one may choose $b$ so that
\begin{equation}\label{dense_f_A}
\mathbb E_{x\in[N]} f_A(x)\ge c\delta
\end{equation}
for some $c>0$ to be determined later. 

Let 
\[X_W=\Big\{b\in \Z_W:\ b\not\equiv -h_i \pmod p \ \text{for every }p\mid W \text{ and every }1\le i\le k\Big\}.
\]
and 
\[\Omega_H(p)=\{b\in \Z_p: b\not\equiv -h_i\pmod{p} \text{ for every }1\leq i\leq k.\}\]
By the Chinese remainder theorem, we may count the elements of $X_W$ in the following way (here $L=\max\{h_1,...,h_k\}$)
\begin{align*}
|X_W|&=W\prod_{p<L}\frac{|\Omega_H(p)|}{p}\prod_{L<p\leq w}\left(1-\frac{k}{p}\right)\\
&\leq C_2W\prod_{1\leq p\leq w}\left(1-\frac{1}{p}\right)^k\\
&=C_2W\left(\frac{\phi(W)}{W}\right)^k,
\end{align*}
where $C_2$ is a constant depending only on $L$. By \eqref{density}, $|A|\geq \delta\frac{C_1NW}{\log^k N}.$ Thus, it is not hard to see that
\[\sum_{b\in\Z_W}|\{n\in[\sqrt{N}, N-\sqrt{N}]: nW+b\in A\}|\geq \delta\frac{C_1NW}{2\log^k N}\]
as well. But, if $b\notin X_W$ and $n\in[\sqrt N,\,N-\sqrt N]$, then
$b+h_i\equiv 0\pmod p$ for some $1\le i\le k$ and some prime $p\mid W$
(in particular, $p\le w$). So 
\[P^{-}(Wn+b+h_i)<w<(nW+b)^{\eps_0}\]
since $n\geq \sqrt{N}$ and $w$ is small. Thus, $Wn+b\not\in A$. Hence, 
\[\sum_{b\in \Z_W}|\{n\in[\sqrt{N}, N-\sqrt{N}]: nW+b\in A\}|=\sum_{b\in X_W}|\{n\in[\sqrt{N}, N-\sqrt{N}]: nW+b\in A\}|.\]
It follows by the pigeonhole principle that there exists $b\in X_W$ such that 
\[|\{n\in[\sqrt{N}, N-\sqrt{N}]: nW+b\in A\}|\geq \delta\frac{C_1}{2C_2}\left(\frac{W}{\phi(W)}\right)^k\frac{N}{\log^k N}.\]
This immediately yields \eqref{dense_f_A} with $c=\frac{c_0C_1}{2C_2}$.

Next we define our pseudorandom measure $\nu_{\mathcal{A}}$ as follows:
\begin{equation}\label{nuA}
\nu_{\mathcal{A}}(x)= \left(\frac{\phi(W) \log R}{W}\right)^k\prod_{1\leq i\leq k}\left( \sum_{d_i | Wx+b+h_i} \mu(d_i)\chi\left( \frac{\log d_i}{\log R}\right) \right)^2,
\end{equation}
where $R=N^{\eta_0}$ for some sufficiently small $\eta_0$ to be chosen later and $\chi:\R\to\R$ is a smooth function compactly supported on $[-1,1]$ with $\chi(0)>1\slash 2$ and $\int_0^1 |\chi'(t)|^{2}dt=1$.  For example, one may take $\chi$ defined by 
\[
\chi(t)=
\begin{cases}
De^{\frac{1}{t^2-1}}&\text{ if } |t|<1\\
0&\text{ if } |t|\geq 1,
\end{cases}
\]
where $D$ is chosen so that $\int_0^1 |\chi'(t)|^{2}dt=1$. 

Crucially, we may choose $c_0$ small enough that $f_A$ is bounded pointwise by $\nu_{\mathcal{A}}$: 
\begin{equation}\label{pwbound}
0\leq f_A\leq \nu_{\mathcal{A}}.
\end{equation}
Indeed, if we assume that $f_A(x)$ is positive for some $x\in \N$, then, by the definition of $\mathcal{A}$, we have
\[P^{-} (Wx+b+h_i)\geq x^{\eps_0}\geq N^{\eps_0\slash 2}>R\]
as long as $\eta_0<\eps_0\slash 2$. Thus, if $d_i>1$ is a divisor of $Wx+b+h_i$, then $d_i> R$, so that $\frac{\log d_i}{\log R}>1$ and $\chi(\log d_i/\log R)=0$. Thus, for each $i$, the only divisor $d_i$ of $Wx+b+h_i$ which contributes to the right-hand side of \eqref{nuA} is $d_i=1$. This gives 
\[\nu_{\mathcal{A}}(x)\geq \frac{\eta_0^k}{4^k}\cdot\left(\frac{\phi(W)\log N}{W}\right)^k,\]
so it suffices to take $c_0<\frac{\eta_0^k}{4^k}$. 

\section{The correlation estimate}\label{correlationestimate}
Here we derive a correlation estimate which is the key to proving the polynomial forms condition for $\nu_{\mathcal{A}}$ (Proposition \ref{polyforms}). Our proof is a close adaptation of the proof of corollary 10.5 in \cite{TZ08}. In particular, for fixed shifts $r_1,\dots,r_J$ we evaluate
\[
\mathbb E_{x\in[N]}\prod_{j=1}^J \nu_{\mathcal A}(x+r_j),
\]
obtaining a main term $1+o(1)$ together with a correction coming from
local obstructions. The analogous correlation estimates over the primes are out of reach. Indeed, if the shifts $\{r_1,...,r_J\}$ form an admissible $J$-tuple, then this is related to the Hardy--Littlewood prime tuples conjecture; this is why we need the model $\nu_{\mathcal{A}}$.  Moreover, we have sieve-theoretic techniques which can be used to analyze $\nu_{\mathcal A}$, in the spirit of the method of Goldston and Y{\i}ld{\i}r{\i}m \cite{GY05}.

We now introduce the corresponding local factors and the notion of good/bad/terrible primes needed to state and prove this correlation estimate. 
\begin{defn}[Local factor]\label{local-def}
Let $P_1,\ldots,P_J\in \Z[x_1,\ldots,x_D]$ be polynomials with integer coefficients.  For any prime $p$, we define
the (principal) \emph{local factor}
$$
c_p(P_1,\ldots,P_J) := \mathbb{E}_{x \in \mathbb{F}_p^D} \prod_{j \in [J]} \1_{P_j(x) \equiv 0 \bmod p}.
$$

\end{defn}

\begin{defn}[Good prime]\label{bad-def} Let $P_1,\ldots,P_J \in \Z[x_1,\ldots,x_D]$ be a collection of polynomials.  We say that a prime $p$ is \emph{good} with respect to $P_1,\ldots,P_J$ if the following hold:
\begin{itemize}
\item The polynomials $P_1 \bmod p,\ldots,P_J \bmod p$ are pairwise coprime.
\item For each $1 \leq j \leq J$, there exists a co-ordinate $1 \leq i_j \leq D$ for which we have the linear behavior
\begin{align*}
 P_j(x_1,\ldots,x_D) &= P_{j,1}(x_1,\ldots,x_{i_j-1}, x_{i_j+1},\ldots, x_{D}) x_{i_j} \\
 &\quad + P_{j,0}(x_1,\ldots,x_{i_j-1},x_{i_j+1},\ldots,x_{D}) \bmod p 
\end{align*}
where $P_{j,1}, P_{j,0} \in \F_p[x_1,\ldots,x_{i_j-1},x_{i_j+1},\ldots,x_D]$ are such that $P_{j,1}$ is non-zero and coprime to $P_{j,0}$.
\end{itemize}
We say that a prime is \emph{bad} if it is not good.  We say that a prime is \emph{terrible} if at least one of the $P_j$ vanish identically modulo $p$ (i.e. all the coefficients are divisible by $p$).  Note that terrible primes are automatically bad.
\end{defn}
We also need the following estimates which follow from some ``elementary'' algebra in \cite[Appendix D]{TZ08}.
\begin{lem}[Local estimates]\label{local-cor}  Let $P_1,\ldots,P_J \in \Z[x_1,\ldots,x_D]$ have degree at most $d$, let $p$ be a prime, and let $S \subset \{1,\ldots,J\}$.  Then:
\begin{itemize}
\item[(a)] If $|S| = 0$, then $c_p( (P_j)_{j \in S} ) = 1$.
\item[(b)] If $|S| \geq 1$ and $p$ is not terrible, then $c_p( (P_j)_{j \in S} ) = O_{d,D,J}(\frac{1}{p})$.
\item[(c)] If $|S| = 1$ and $p$ is good, then $c_p( (P_j)_{j \in S} ) = \frac{1}{p} + O_{d,D,J}(\frac{1}{p^2})$.
\item[(d)] If $|S| > 1$ and $p$ is good, then $c_p( (P_j)_{j \in S} ) = O_{d,D,J}(\frac{1}{p^2})$.
\end{itemize}
\end{lem}

With these definitions and local estimates in hand, we can state and prove the correlation estimate. 
\begin{prop}[Correlation Estimate]\label{correlationestimateprop}Let $r_1,\ldots,r_J$ be integers, and let  $ N \geq R^{4kJ+1}$.  Then
\begin{equation}\label{estimate}
 \EE_{x \in [N]} \prod_{j \in [J]} \nu_{\mathcal{A}}(x+r_j) = 1 + o(1) + O\left( \EXP\left( O\left( \sum_{p \in \PP_b} \frac{1}{p} \right) \right) \right)
\end{equation}
where $\PP_b$ is the set of primes $w< p\leq R^{\log R}$ which are bad with respect to the polynomials $P_{ij}=W(x+r_j)+b+h_i$ for $(i,j)\in [k]\times [J].$ 
\end{prop}
Notice that in our case, since $P_{ij}$ is linear, 
\begin{align*}
\PP_b &= \{w < p \leq R^{\log R}: p |(Wr_j+h_i)-(Wr_{j'}+h_{i'}) \hbox{ for some } (i,j)\neq (i',j')\}.
\end{align*}
Also, in order for the condition $N\geq R^{4kJ+1}$ to hold, we must take $\eta_0\leq (4kJ+1)^{-1}$. 

\begin{proof}
    To simplify notation, set $S=\left( \frac{\phi(W) \log R}{W} \right)^{kJ}.$   Expanding the left-hand side of \eqref{estimate} gives 
    \begin{align*}
        &S\cdot \EE_{x \in [N]} \prod_{i,j}\left( \sum_{m_{ij} | W(x + r_j)+b+h_i} \mu(m_{ij})\chi\left( \frac{\log m_{ij}}{\log R}\right) \right)^2 \notag\\
        &= S\cdot \EE_{x \in [N]} \prod_{i,j} \sum_{[m_{ij}, m_{ij}'] | W(x + r_j)+b+h_i} \mu(m_{ij})\mu(m_{ij}' )\chi\left( \frac{\log m_{ij}}{\log R}\right) \chi\left( \frac{\log m_{ij}'}{\log R}\right)  \notag\\
        \end{align*}
By expressing the condition $[m_{ij}, m_{ij}'] | W(x + r_j)+b+h_i$ with an indicator function in the summand and expanding the product, this becomes
\begin{align*}
&S\cdot \EE_{x \in [N]} \prod_{i,j} \sum_{\substack{{m_{ij}, m_{ij}' \geq 1} \\ {(i, j) \in [k]\times [J]}}} \mu(m_{ij})\mu(m_{ij}' )\chi\left( \frac{\log m_{ij}}{\log R}\right) \chi\left( \frac{\log m_{ij}'}{\log R}\right)  \\
&\;\;\;\;\;\;\;\;\;\;\;\;\;\;\;\;\;\;\;\;\;\;\;\;\;\;\;\;\;\;\;\;\;\;\;\;\;\;\;\;\;\;\;\;\;\;\;\;\;\;\;\;\;\;\;\;\;\;\;\;\;\;\;\;\;\;\;\;\;\;\;\;\;\;\;\;\;\;\times\1_{[m_{ij}, m_{ij}']| W(x+r_j) + b + h_i}.\\
&= S\cdot \EE_{x \in [N]}  \sum_{\substack{{m_{ij}, m_{ij}' \geq 1} \\ {(i, j) \in [k]\times [J]}}}\prod_{i,j}\mu(m_{ij})\mu(m_{ij}' )\chi\left( \frac{\log m_{ij}}{\log R}\right) \chi\left( \frac{\log m_{ij}'}{\log R}\right) \\
&\;\;\;\;\;\;\;\;\;\;\;\;\;\;\;\;\;\;\;\;\;\;\;\;\;\;\;\;\;\;\;\;\;\;\;\;\;\;\;\;\;\;\;\;\;\;\;\;\;\;\;\;\;\;\;\;\;\;\;\;\;\;\;\;\;\;\;\;\;\;\;\;\;\;\;\;\;\;\times\prod_{i,j}\1_{[m_{ij}, m_{ij}']| W(x+r_j) + b + h_i},\\
\end{align*}
where we also used the fact that, by the support properties of $\mu$ and $\chi$, we may restrict the sum to square free $m_{ij},m'_{ij}$ with $m_{ij},m'_{ij}<R$ for all $i$ and $j$, so that the sums and products above are finite. Finally, by utilizing linearity of the average over the $x$ variable, this is 
\begin{align}
& S\cdot \sum_{\substack{{m_{ij}, m_{ij}' \geq 1} \\ {(i, j) \in [k]\times [J]}}} \prod_{i,j}\mu(m_{ij})\mu(m_{ij}') \chi\left( \frac{\log m_{ij}}{\log R}\right) \chi\left( \frac{\log m_{ij}'}{\log R} \right) \notag\\
&\;\;\;\;\;\;\;\;\;\;\;\;\;\;\;\;\;\;\;\;\;\;\;\;\;\;\;\;\;\;\;\;\;\;\;\;\;\;\;\;\;\;\;\;\;\;\;\;\;\;\times\EE_{x \in [N]} \prod_{i,j} \1_{[m_{ij}, m_{ij}']| W(x+r_j) + b + h_i}.\label{expandedstep}
\end{align}
 By \cite{TZ08} corollary C.3, \begin{align*}
    &\EE_{x \in [N]} \prod_{i,j} \1_{[m_{ij}, m_{ij}']| W(x+r_j) + b + h_i} \\
    &= \left( 1 + O \left( \frac{M}{N} \right)\right) \EE_{y \in \Z_{M}} \prod_{i,j} \1_{[m_{ij}, m_{ij}']| W(y+r_j) + b + h_i}
\end{align*}
where $\displaystyle{M}$ is the least common multiple of $[m_{ij}, m_{ij}']$ for all $(i,j)\in [k]\times [J]$. In particular, $M \leq R^{2kJ}$, so 
\begin{align*}
    &\EE_{x \in [N]} \prod_{i,j} \1_{[m_{ij}, m_{ij}']| W(x+r_j) + b + h_i} \\
    &= \left( 1 + O \left( \frac{1}{R^{2kJ + 1}} \right)\right) \EE_{y \in \Z_{M}} \prod_{i,j} \1_{[m_{ij}, m_{ij}']| W(y+r_j) + b + h_i}.
\end{align*}
To get rid of the $\displaystyle{ O \left( \frac{1}{R^{2kJ + 1}} \right)}$, notice that $\displaystyle{\EE_{y \in \Z_{M}} \prod_{i,j} \1_{[m_{ij}, m_{ij}']| W(y+r_j) + b + h_i}}$ can be crudely bounded above by $1$, so the contribution of $\displaystyle{ O \left( \frac{1}{R^{2kJ + 1}} \right)}$ to \eqref{expandedstep} is \begin{align*}
    O\left( S\cdot \sum_{\substack{{1 \leq m_{ij}, m_{ij}' \leq R} \\ {(i, j) \in [k]\times [J]}}} \frac{1}{R^{2kJ+1}} \right) &= O \left( \frac{\log^{kJ}R}{R} \right) = o(1).
\end{align*}

The local factor $\displaystyle{\alpha_{lcm([m_{ij}, m_{ij}']),(i,j) \in [k]\times [J])} = \EE_{y \in \Z_{M}} \prod_{i,j} \1_{[m_{ij}, m_{ij}']| W(y+r_j) + b + h_i} }$ is multiplicative by the Chinese Remainder Theorem so that if $\displaystyle{[m_{ij}, m_{ij}'] = \prod_{p}p^{r_{ij}(p)}}$, for $\displaystyle{(i,j) \in [k]\times [J])}$, then $\alpha_{lcm([m_{ij}, m_{ij}'])} = \prod_{p} \alpha_{(p^{r_{ij}})_{ij}}$. Here, the $r_{ij}$'s depend on $p$ and are either $0$ or $1$ when $m_{ij}, m_{ij}'$ are squarefree.
Now \begin{align*}
    \alpha_{(p^{r_{ij}})_{ij}} &= \EE_{y \in \Z_p} \prod_{i,j} \1_{p^{r_{ij}} | W(y + r_j)+ b + h_{i}} = c_p\left( \left( W\left(y + r_j\right) + b + h_i \right)_{r_{ij} = 1} \right),
\end{align*}
where $\displaystyle{c_p(P_1,...,P_L) = \EE_{x \in \F_p} \prod_{l \in [L]} \1_{P_l(x) \equiv 0 (p)}}$  and $\left( W\left(y + r_j\right) + b + h_i \right)_{r_{ij} = 1}$ is a tuple of polynomials consisting of $W\left(y + r_j\right) + b + h_i$ for each pair $(i,j)$ such that $r_{ij}=1$. (If no such $(i,j)$ exist, this tuple is empty and $c_p(\emptyset)=1$.)
Therefore \begin{align*}
    \alpha_{lcm([m_{ij}, m_{ij}'])} = \prod_{p} c_p\left( \left( W\left(y + r_j\right) + b + h_i \right)_{r_{ij} = 1} \right).
\end{align*}

We now replace the $\chi$ factors with terms that are multiplicative in the $m_{ij}, m_{ij}'$. We have \begin{equation}\label{phidef}
    e^x \chi(x) = \int_{-\infty}^{\infty} \varphi(\xi) e^{-ix \xi} d\xi,
\end{equation} 
as $\chi$ is smooth and compactly supported, for some smooth and rapidly decreasing $\varphi$ (in particular $\varphi(\xi) = O_{A}((1 + |\xi|)^{-A})$ for any $A>0$). So from the hypothesis on $\chi$ we have \begin{equation}
    \int_{-\infty}^\infty \int_{-\infty}^\infty \frac{(1+it) (1+it')}{2+it+it'} \varphi(t) \varphi(t')\ dt dt' = 1.
\end{equation}

From the above Fourier expansion we have \begin{align*}
    \chi \left( \frac{\log m_{ij}}{\log R} \right) &= \int_{- \infty}^{\infty} m_{ij}^{-z_{ij}} \varphi(\xi_{ij}) d\xi_{ij} \\
    \chi \left( \frac{\log m_{ij}'}{\log R} \right) &= \int_{- \infty}^{\infty} (m_{ij}')^{-z_{ij}'} \varphi(\xi_{ij}') d\xi_{ij}'.
\end{align*}
where \begin{equation}\label{zdef}
    z_{ij} := \frac{1 + i\xi_{ij}}{\log R}, \hspace{2mm} z_{ij}' := \frac{1 + i\xi_{ij}'}{\log R}.
\end{equation}

With these improvements, the expanded step is now reduced to \begin{align*}
    &S\cdot \sum_{\substack{{m_{ij}, m_{ij}' \geq 1} \\ {(i, j) \in [k]\times [J]}}} \int_{-\infty}^{\infty} \ldots \int_{-\infty}^{\infty} \prod_{i,j} \mu(m_{ij})\mu(m_{ij}') m_{ij}^{-z_{ij}} (m_{ij}')^{-z_{ij}'} \varphi(\xi_{ij}) \varphi(\xi_{ij}')d\xi_{ij}d\xi_{ij}' \\
    &\;\;\;\;\;\;\;\;\;\;\;\;\;\;\;\;\;\;\;\;\;\;\;\;\;\;\;\;\;\;\;\;\;\;\;\;\;\;\;\;\;\;\;\;\;\;\;\;\;\;\times\prod_{p \leq R^{\log R}} c_p\left( \left( W\left(y + r_j\right) + b + h_i \right)_{r_{ij} = 1} \right) + o(1).
\end{align*}
The main term is then reduced to \begin{align*}
    \int_{-\infty}^{\infty} \ldots \int_{-\infty}^{\infty} S\cdot\prod_{p \leq R^{\log R}} E_{p} \prod_{i, j} \varphi(\xi_{ij}) \varphi(\xi_{ij}')d\xi_{ij}d\xi_{ij}'
\end{align*}
where $E_p = E_p(z_{11},z_{11}',...,z_{kJ},z_{kJ}')$ is the Euler factor \begin{align*}
    E_p := \sum_{\substack{{m_{ij}, m_{ij}' \in \{ 1, p\}} \\ \forall (i, j) \in [k] \times [J]}} \prod_{i, j} \mu(m_{ij})\mu(m_{ij}') m_{ij}^{-z_{ij}} (m_{ij}')^{-z_{ij}'}c_p\left( \left( W\left(y + r_j\right) + b + h_i \right)_{r_{ij} = 1} \right).
\end{align*}
To approximate $E_p$, we introduce $E_p'$ \begin{align*}
    E_p' := \prod_{i, j} \frac{\left( 1 - \frac{1}{p^{1 + z_{ij}}} \right)\left( 1 - \frac{1}{p^{1 + z_{ij}'}} \right)}{\left( 1 - \frac{1}{p^{1 + z_{ij} + z_{ij}'}} \right)}.
\end{align*}

\begin{lem}[Euler product estimate]\label{EPestimate} We have \begin{align*}
    \prod_{p \leq R^{\log R}} \frac{E_p}{E_p'} = \left( 1 + o(1) + O \left( \EXP \left( O \left( \sum_{p \in \PP_b} \frac{1}{p} \right) \right) \right) \right) \left( \frac{W}{\phi(W)} \right)^{kJ}.
\end{align*}
    
\end{lem}
\begin{proof}
    We first prove that if $p < w$, then
    \begin{equation}\label{p<w}
        E_p = 1 , \hspace{1cm} E_p' = \left( 1 - \frac{1}{p} \right)^{kJ} + o(1).
    \end{equation}
To see why $E_p=1$ for $p<w$, recall that we chose $b$ so that $b\not\equiv -h_i\pmod{W}$ for any $1\leq i\leq k$. Thus, since $p|W$ for any prime $p<w$, we have 
\[ W\left(y + r_j\right) + b + h_i\not\equiv0\pmod{p}\]
so that if $r_{ij}=1$ for some $i,j$, then 
\[c_p\left( \left( W\left(y + r_j\right) + b + h_i \right)_{r_{ij} = 1} \right)=0.\]
It follows that the only time the summand in the definition of $E_p$ is nonzero is when $r_{ij}=0$ for all $i$ and $j$. This is exactly when $m_{ij}=1$ and $m_{ij}'=1$ for all $i$ and $j$, in which case  $\left( W\left(y + r_j\right) + b + h_i \right)_{r_{ij} = 1}$ is empty so that 
\[c_p\left( \left( W\left(y + r_j\right) + b + h_i \right)_{r_{ij} = 1} \right)=1\] 
and
\begin{align*}
\mu(m_{ij})\mu(m_{ij}') m_{ij}^{-z_{ij}} (m_{ij}')^{-z_{ij}'}c_p\left( \left( W\left(y + r_j\right) + b + h_i \right)_{r_{ij} = 1} \right)
&=1
\end{align*}
for all $i$ and $j$ and thus $E_p=\prod_{i,j}1=1.$ To see the why $E_p'=\left( 1 - \frac{1}{p} \right)^{kJ} + o(1)$ for $p<w$, recall from the definitions of $z_{ij}$ and $z_{ij}'$ \eqref{zdef} that $p^{-z_{ij}}, p^{-z_{ij}'}=1+o(1)$ for fixed $\xi_{ij}, \xi{ij}'$. Here, we also used the fact that $p<w$ and $w=\log\log\log N'$ is much smaller than $R=N^{\eta_0}$. 
Thus, for each $i$ and $j$, we have
\[
    \frac{\left( 1 - \frac{1}{p^{1 + z_{ij}}} \right)\left( 1 - \frac{1}{p^{1 + z_{ij}'}} \right)}{\left( 1 - \frac{1}{p^{1 + z_{ij} + z_{ij}'}}\right) }=1-\frac{1}{p}+o(1).
\]
and therefore 
\[E_p'=\prod_{i,j}\left(1-\frac{1}{p}+o(1)\right)=\left( 1 - \frac{1}{p} \right)^{kJ} + o(1)\]

With these two equations \eqref{p<w} in hand,
\begin{align*}
    \prod_{p < w}\frac{E_p}{E_p'} &= (1 + o(1)) \prod_{p<w} \left( 1 - \frac{1}{p} \right)^{-kJ} = (1 + o(1)) \left( \frac{W}{\phi(W)} \right)^{kJ}.
\end{align*}
To estimate $\displaystyle{\prod_{w \leq p \leq R^{\log R}} \frac{E_p}{E_p'}}$, we divide primes in $w \leq p \leq R^{\log R}$ into the good, bad, and terrible primes with respect to the polynomials $P_{ij}$. The bad primes $w \leq p \leq R^{\log R}$ are in $\PP_b$, and we clearly have no terrible primes $p\geq w$, since in this case $p$ does not divide $W$. 

    If $p$ is bad, $\displaystyle{c_p((W(y + r_j) + b + h_i)_{r_{ij} = 1}) = O\left( \frac{1}{p}\right)}$ from Lemma 9.5 of \cite{TZ08}. So, \begin{equation*}
        E_p = 1 + O \left( \frac{1}{p} \right) = \exp \left( O\left(\frac{1}{p}\right)\right)E_p'
    \end{equation*} and \begin{align*}
        \prod_{\substack{{w \leq p \leq R^{\log R}}\\ { p-\text{bad}}}} \frac{E_p}{E_p'} = \prod_{\substack{{w \leq p \leq R^{\log R}}\\ { p-\text{bad}}}}\exp \left( O\left(\frac{1}{p}\right)\right) &=\exp\left( \sum_{\substack{{w \leq p \leq R^{\log R}}\\ { p-\text{bad}}}} O \left( \frac{1}{p} \right) \right) \\
        &= 1 + O \left( \EXP \left( O \left( \sum_{p \in \PP_b} \frac{1}{p} \right) \right) \right).
    \end{align*}
    
    For good $p$, again using Lemma 9.5 of \cite{TZ08}, we have 
        $E_p = \left( 1 + O \left( \frac{1}{p^2} \right) \right)E_p',$ and \begin{align*}
        \prod_{\substack{{w \leq p \leq R^{\log R}}\\ { p-\text{good}}}} \frac{E_p}{E_p'} &= \prod_{\substack{{w \leq p \leq R^{\log R}}\\ { p-\text{good}}}} \left( 1 + O \left( \frac{1}{p^2} \right) \right) = 1 + o(1).
    \end{align*}
    as $w \rightarrow \infty$. Putting all the cases of $p$ together, we have the result.
\end{proof}

Now we complete the proof of Proposition \ref{correlationestimateprop}. Using some standard facts about the Riemann Zeta function (specifically Theorems 11.7 and 12.5 from Apostol \cite{Apostol}) we can write \begin{align*}
    \prod_{p \leq R^{\log R}} E_p' = \prod_{i, j} \left( 1 + o(1) \right) \frac{\zeta(1 + z_{ij} + z_{ij}')}{\zeta(1 + z_{ij})\zeta(1 + z_{ij}')},
\end{align*}
and we also have \begin{align*}
    \frac{1}{\zeta(1+(1+i\xi)/\log R)} &= (1 + o( (1 + |\xi|)^2 ) ) \frac{1+i\xi}{\log R}, \\
    \zeta(1+(1+i\xi)/\log R) &= (1 + o( (1 + |\xi|)^2 ) ) \frac{\log R}{1+i\xi}
\end{align*}
for any real $\xi$, so we get \begin{align*}
    \prod_{p \leq R^{\log R}} E'_p =
\prod_{i, j}(1 + o( (1 + |\xi_{ij}| + |\xi'_{ij}|)^{6} ) )
\frac{1}{\log R} \frac{(1 + i \xi_{ij}) (1 + i \xi'_{ij})}{2 + i \xi_{ij} + i \xi'_{ij}}.
\end{align*}

Applying Lemma \ref{EPestimate} gives 
\begin{align*}
    &S\cdot\prod_{p \leq R^{\log R}} E_{p} = \left( 1 + o\left( \prod_{i,j} (1 + |\xi_{ij}| + |\xi'_{ij}|)^{6}   \right) + O \left( \EXP \left( O \left( \sum_{p \in \PP_b} \frac{1}{p} \right) \right) \right) \right)\\
    &\;\;\;\;\;\;\;\;\;\;\;\;\;\;\;\;\;\;\;\;\;\;\;\;\;\;\;\;\;\;\;\times\prod_{i,j} \frac{(1 + i \xi_{ij}) (1 + i \xi'_{ij})}{2 + i \xi_{ij} + i \xi'_{ij}}.
\end{align*}
Thus, to prove  Proposition \ref{correlationestimateprop}, it is enough to show that \begin{align*}
     \int_{-\infty}^{\infty} \dots \int_{-\infty}^{\infty} \prod_{i, j} \varphi(\xi_{ij}) \varphi(\xi'_{ij}) \frac{(1 + i \xi_{ij}) (1 + i \xi'_{ij})}{2 + i \xi_{ij} + i \xi'_{ij}}\ d\xi_{ij} d\xi'_{ij} = 1.
\end{align*} 
and \begin{align*}
    \int_{-\infty}^{\infty} \dots \int_{-\infty}^{\infty} \prod_{i, j} |\varphi(\xi_{ij})| |\varphi(\xi'_{ij})| (1 + |\xi_{ij}| + |\xi'_{ij}|)^6 \frac{|1 + i \xi_{ij}| |1 + i \xi'_{ij}|}{|2 + i \xi_{ij} + i \xi'_{ij}|}\ d\xi_{ij} d\xi'_{ij} = O(1).
\end{align*}
We get the first equation from \eqref{phidef} and the fact that $\int_0^1 |\chi'(t)|^{2}dt=1$, and the second equation follows from the rapid decay of $\varphi$.
\end{proof}

\section{The polynomial forms condition}\label{polynomialforms}
Now that we have the correlation estimate, we can establish the fundamental pseudorandomness property for $\nu_{\mathcal{A}}$, the polynomial forms condition as stated in \cite{Lott}. This shows $\nu_{\mathcal A}$ behaves like the constant function $1$ for multilinear polynomial averages. In \cite{Lott}, this is used to effectively carry out PET induction to prove the transference theorem. 

This particular formulation allows us to circumvent the convex geometry in \cite{TZ08, TZ14}. However, the ideas in the proof are still essentially the same as in \cite{TZ08, TZ14}.
\begin{prop}[Polynomial forms condition]\label{polyforms}  Let $d,D,J$ be fixed natural numbers, let $\eps_0$ sufficiently small and $L$ sufficiently large depending on $d,D,J$. Let $Q_1,...,Q_J\in \Z[m_1,...,m_d]$ be polynomials of degree at most $D$, with coefficients of size $O(W^C)$. Assume $Q_i-Q_j$ is nonconstant for all $i\neq j$. Then
\eq\label{2.8}
\EE_{\vec \ell \in [H]^d}\ \EE_{x\in [N]} \prod_{j=1}^J \nu_{\mathcal{A}}(x+Q_j(\vec \ell)) 
= 1+o(1),
\ee
where $H=\log^{\sqrt{L}}N$. 
\end{prop}
\begin{proof}
By the correlation estimate, for each fixed $\vec \ell$ we have
\[\EE_{x\in [N]} \prod_{j = 1}^J \nu_{\mathcal{A}}(x+Q_j(\vec \ell))= 1 + o(1) + O\left( \EXP\left( O\left( \sum_{p \in \PP_b(\vec\ell)} \frac{1}{p} \right) \right) \right).\]
where 
\begin{align*}
&\PP_b(\vec\ell) \\
&= \{w < p \leq R^{\log R} : p \mid (WQ_j(\vec\ell) + h_i) - (WQ_{j'}(\vec\ell) + h_{i'}) \text{ for some } (i,j) \neq (i',j')\} \\
&=\bigcup_{1\leq i,i'\leq k} \{w < p \leq R^{\log R} : p \mid (WQ_j(\vec\ell) + h_i) - (WQ_{j'}(\vec\ell) + h_{i'}) \text{ for some } j \neq j'\},
\end{align*}
where we made the observation that if $j = j'$, then $p \mid (WQ_j(\vec\ell) + h_i) - (WQ_{j'}(\vec\ell) + h_{i'})$ forces $p \mid h_i - h_{i'}$, which in turn forces $i = i'$, since $p > w > \max |h_i - h_{i'}|$ for large enough $N$. Thus, it suffices to show 
\[\EE_{\vec \ell\in [H]^d} O\left( \EXP\left( O\left( \sum_{p \in \PP_b(\vec\ell)} \frac{1}{p} \right) \right) \right)=o(1).\]
 Let 
\[S(i,j,i',j', \vec \ell):=\sum_{\substack{w < p \leq R^{\log R} \\ p \mid (WQ_j(\vec\ell) + h_i) - (WQ_{j'}(\vec\ell) + h_{i'})}} \frac{1}{p}.\]
Notice that 
\begin{align*}
\sum_{p \in \PP_b(\vec\ell)} \frac{1}{p} 
=\sum_{(i,j) \neq (i',j')} S(i,j,i',j',\vec\ell) \leq\sum_{1\leq i,i'\leq k} \sum_{1 \leq j < j' \leq J} S(i,j,i',j',\vec\ell) .
\end{align*}
Thus, 
\begin{align*}
&\EE_{\vec \ell\in [H]^d} 
O\left( 
\EXP\left( 
O\left( 
\sum_{p \in \PP_b(\vec\ell)} \frac{1}{p} 
\right) 
\right) 
\right) \\
&= \EE_{\vec \ell\in [H]^d} 
O\left( 
\EXP\left( 
O\left( 
\sum_{1 \leq i, i' \leq k} 
\sum_{1 \leq j < j' \leq J} 
S(i,j,i',j',\vec\ell)
\right) 
\right) 
\right) \\
&\ll \sum_{1 \leq i, i' \leq k} 
\sum_{1 \leq j < j' \leq J} 
\EE_{\vec \ell\in [H]^d} 
O\left( 
\EXP\left( 
O(S(i,j,i',j', \vec\ell)
\right) 
\right),
\end{align*}
where we repeatedly applied the bound $\EXP(a+b)\ll \EXP(2a)+\EXP(2b)$. Thus, it suffices to show 

\begin{equation}
\EE_{\vec \ell\in [H]^d} 
\EXP\left( 
O\left( S(i,j,i',j',\vec\ell)
\right) 
\right) 
=
o(1), \label{NastySum}
\end{equation}
for each fixed $1\leq i,i'\leq k$ and $1\leq j<j'\leq J$.

To estimate the \eqref{NastySum} first consider set $B$ of the primes $p\geq w$ for which the polynomial $\De Q(\vec\ell):=(WQ_j(\vec\ell) + h_i) - (WQ_{j'}(\vec\ell) + h_{i'})$ identically vanishes$\pmod{p}$. Since all such primes $p$ must divide the coefficients of $\De Q$ which are of size $O(W^C)$ we have
\[w^{|B|}\leq \prod_{p\in B} p \ll W^C \ll e^{C\,w},\]
by the prime number theorem, thus 
\[\sum_{p\in B}\ \frac{1}{p}\, \leq\, \frac{|B|}{w}\, \ll \frac{1}{\log\,w}=o(1),\]
so these primes have negligible contribution to \eqref{NastySum}.

By \cite[Lemma E.1]{TZ08}, we have the following for any set $E$ of primes: 
\[\EXP\left(O\left(\sum_{p\in E}\frac{1}{p}\right)\right)\ll \sum_{p\in E} \frac{\log^{O(1)}p}{p}.\]
Thus, it suffices to prove 
\begin{equation}\label{TidySum}
\EE_{\vec\ell\in [H]^d} \sum_{\substack{w < p \leq R^{\log R}, p\not\in B, p \mid \De Q(\vec\ell)}} \frac{\log^{O(1)}p}{p}=o(1).
\end{equation}

Next, we partition the sum in $p\notin B$ into dyadic intervals $2^s\leq p<2^{s+1}$, and estimate the contribution of each part depending the size of the parameter $s\in \N$.

Assume first that $2^{s/2}\leq \log\,H = \sqrt{L}\,\log\log\,N$. Since the polynomial $\De Q$ has degree at least 1 $\pmod{p}$, and $p$ is much smaller than $H$ when $p<2^{s+1}$, the number of $\vec \ell\in [H]^d$ satisfying $p|\De Q(\vec \ell)$ is at most $H^d/p+O(1)$, thus 
\eq
\EE_{\vec \ell\in [H]^d} \sum_{2^s\leq p < 2^{s+1}\,p\not\in B,\, p|\De Q(\vec \ell)}\ 
\frac{\log^{O(1)}p}{p}\ \leq\ \frac{s^{O(1)}}{2^s} + O(H^{-1\slash 2}). \label{bound1}
\ee

Assume now $2^{s/2} \geq \log H$. Then left side of \eqref{bound1} is estimated by
\begin{align}
\frac{s^{O(1)}}{2^s}&\ \EE_{\vec\ell\in [H]^d}\,|\{2^s\leq p <2^{s+1};
\ p|\De Q(\vec \ell), p\not\in B\}|\notag\\
 &\leq \frac{s^{O(1)}D\log\,H}{\,2^s}+s^{O(1)}O(H^{-1})\notag\\
&\ll \frac{s^{O(1)}}{2^{s/2}}+O(H^{-1\slash 2}). \label{bound2}
\end{align}
 Here we used the well-known fact that $|\{\vec\ell\in [H]^d, \De Q(\vec\ell)=0\}|=O(H^{d-1})$, and if $1\leq |\De Q(\vec \ell)| \leq H^D$, then 
$\ \prod_{p|\De Q(\vec\ell)} p\leq H^D\ $, so the number of primes $p\geq 2^s$, $p|\De Q(\vec \ell)$ is $O(D\log H/s)$.

Therefore, by \eqref{bound1} and \eqref{bound2}, we can sum \eqref{TidySum} over dyadic intervals to conclude that it is $O(w^{-1\slash4})=o(1)$. This completes the proof. 

\end{proof}

\section{The transference principle}\label{transferenceprinciple}

Now we record the formal deduction of Theorem~1 from the transference principle, but first we must state the transference principle, which is a combination of corollaries $3.1$ and $4.2$ in Lott-Magyar-Petridis-Pintz \cite{Lott}. 

\begin{thmD}\label{cor4.2} Let $f:[N]\to\R$ satisfy $0\leq f\leq\nu,$ where $\nu$ satisfies the polynomial forms condition, and let $\eps>0$. Fix polynomials $P_1,...,P_t\in\Z[y]$ with $P_j(0)=0$ for every $1\leq j\leq t$. Then there exists a function $0\leq g\leq 1$ such that
\[
\left|\Lambda_{P_1,....,P_t}f-\Lambda_{P_1,....,P_t}g\right|< \epsilon\hbox{ and }\left|\EE_{x\in[N]}f(x)-\EE_{x\in[N]}g(x)\right|<\epsilon.\]
\end{thmD}
Now the proof of Theorem 1 is straightforward. 
\begin{proof}[Proof of Theorem 1] Given $A\subs \mathcal{A}\cap [N']$ with $|A|\geq \de\,|\mathcal{A}\cap [N']|$ we choose $W$, $N=[N'/W]$, $b$ and the function $f_A$ as in section $4$. Thus, we have \eqref{dense_f_A} and \eqref{pwbound}:

\[
\EE_{x\in[N]}\, f_A(x)\geq \de',\quad 0\leq f_A(x) \leq\nu_{\mathcal{A}} (x),
\]
with $\de'= c\delta$, and $\nu_{\mathcal{A}}$ satisfies the polynomial forms condition by Proposition \ref{polyforms}. Next, we have $Q_j(y):=\frac{P_j(Wy)}{W}\in\Z[y]$
(since $P_j(0)=0$). We apply Theorem C and Theorem D with $Q_1,\dots,Q_t$ in place of $P_1,\dots,P_t$.

By theorems $C$ and $D$ with $\eps>0$ sufficiently small with respect to $\delta'$ and $c(\de'\slash 2)$, we have a bounded $g:[N]\to [0,1]$ such that 
\[
\La_{Q_1,...,Q_t}f_A \geq \La_{Q_1,...,Q_t}g-
\eps-o(1) \geq c(\de'\slash 2)-\eps-o(1)\geq 
c(\de'\slash 2)/2-o(1)>0,
\]
as long as $L\geq L(Q_1,...,Q_t)$ and $N\geq N(\de,Q_1,...,Q_t)$ is sufficiently large. Then $\Lambda_{Q_1,\dots,Q_t}f_A>0$ yields $x,y\in \N$ such that $Wx+b+P_j(Wy)\in A$ for all $j$.
Setting $x_0:=Wx+b$ and $y_0:=Wy$, we obtain $x_0+P_j(y_0)\in A$ for all $j$, thus completing the proof.
\end{proof}

\section{Acknowledgments} The authors would like to thank Ákos Magyar for suggesting the problem.

\end{document}